\documentclass[11pt, reqno]{amsart}
\usepackage{amsthm,amsmath,amsfonts,amssymb,color}
\usepackage[bookmarks]{hyperref}
\reversemarginpar

\newtheorem{thm}{Theorem}[section]

\newtheorem{cor}[thm]{Corollary}
\newtheorem{lemma}[thm]{Lemma}

\newtheorem{preremark}[thm]{Remark}
\newenvironment{remark}{\begin{preremark}\rm}{\medskip \end{preremark}}

\numberwithin{equation}{section}

\newcommand{\norm}[1]{\left\Vert#1\right\Vert}
\newcommand{\abs}[1]{\left\vert#1\right\vert}
\newcommand{\set}[1]{\left\{#1\right\}}

\newcommand{\R}{\mathbb R}
\newcommand{\F}{\mathbb F}
\DeclareMathOperator{\Vol}{Vol}

\newcommand{\eps}{\varepsilon}

\newcommand{\grad} {\nabla}

\newcommand{\dd} {\mathrm{d}}

\DeclareMathOperator{\Def}{Def}
\DeclareMathOperator{\Ric}{Ric}

\def\H{\mathbb H^{2}(-a^{2})}
\def\be{\begin{equation}}
\def\ee{\end{equation}}
\def\qand{\quad\mbox{and}\quad}

\begin{document}
\title[Liouville theorems on a hyperbolic space]{Liouville theorems for the Stationary Navier Stokes equation on a hyperbolic space.}

\author[Chan]{Chi Hin Chan}
\address{Department of Applied Mathematics, National Chiao Tung University,1001 Ta Hsueh Road, Hsinchu, Taiwan 30010, ROC}
\email{cchan@math.nctu.edu.tw}
\author[Czubak]{Magdalena Czubak}
\address{Department of Mathematical Sciences, Binghamton University (SUNY),
Binghamton, NY 13902-6000, USA}
\email{czubak@math.binghamton.edu}

\begin{abstract}
The problem for the stationary Navier-Stokes equation in 3D under finite Dirichlet norm is open.  In this paper we answer the analogous question on the 3D hyperbolic space.   We also address other dimensions and more general manifolds.
\end{abstract}
\date{\today}
\maketitle

%\tableofcontents

 \section{Introduction}
 Consider the following stationary Navier-Stokes equation on $\R^n$, $n\geq 2$
 \begin{align}
 -\Delta u +u\cdot \nabla u + \grad p&=0,\\
 \nabla \cdot u&=0,
 \end{align}
 together with the conditions that
 \begin{align}
  \lim_{\abs{x}\rightarrow \infty} u(x)=0\qand \int_{\R^n}\abs{\nabla v}^2 <\infty,
 \end{align}
 where
 \[
 u: \R^n \rightarrow \R^n\qand p: \R^n \rightarrow \R.
 \]
 The zero solution is a solution, but is it the only solution?  In all dimensions $n\neq 3$, the answer is known, and it is yes, zero solution is the only solution (see for example \cite{Galdi}).   In three dimensions, Galdi \cite{Galdi} has shown that if one imposes in addition that $u \in L^{\frac 92}(\R^3)$ that the answer is also yes.  However, the full problem in three dimensions remains open.

In this paper, under the assumptions of finite $\dot H^1$ norm only, without any additional assumptions on the integrability, we give a positive answer on a hyperbolic space in three dimensions (as well as four dimensions).  The main result is
\begin{thm}\label{MAINTHM}
Let $u$ be a divergence free, smooth $1$-form on $\mathbb{H}^N(-a^2)$, and $p \in C^{\infty}(\mathbb{H}^N(-a^2))$, with $N\geq 2$. If $(u, p)$ satisfy the following stationary Navier-Stokes equation on $\mathbb{H}^ N(-a^2)$
\begin{equation}\label{NSequation}
\begin{split}
2 \Def^* \Def u + \nabla_u u + \dd p & = 0 ,\\
\dd^* u & = 0,
\end{split}
\end{equation}
and if
 \begin{equation}\label{FiniteDirichlet}
\int_{\mathbb{H}^N(-a^2)} \big | \nabla u \big |^2\Vol_{\mathbb{H}^N(-a^2)} < \infty ,
\end{equation}
then $u = 0$ on $\mathbb{H}^N(-a^2)$ if $N=3, 4$.  If $N=2$ and we know in addition that $u\in L^\infty(\mathbb{H}^2(-a^2))$, then $u=\dd F\in L^2(\H)$, where $F$ is a harmonic function.  If $N\geq 5$, and in addition $u\in L^\infty(\mathbb H^N(-a^2)),$  then $u = 0$ on $\mathbb{H}^N(-a^2)$.
\end{thm}
\begin{remark}
In the Euclidean setting it is known (see for example \cite[Ch 3, Proposition 2.7]{SereginBook}) that solutions to the stationary Navier-Stokes problem are smooth.  Hence the smoothness assumption in Theorem \ref{MAINTHM} is quite natural.  Moreover, if we consider the assumption  $\lim_{\abs{x}\rightarrow \infty} u(x)=0$ for solutions on a Euclidean space, then this combined with smoothness, leads to $L^\infty$ bound on the solution.  Hence the $L^\infty$ assumption for dimensions $N=2$ and $N\geq 5$ is not surprising either.
\end{remark}
For convenience, we work with $1$-forms instead of vector fields.  Using the metric, one can easily move between one and the other.  Here $\nabla$ denotes the covariant derivative, and $\Def$ is the deformation tensor, which is the symmetrization of the covariant derivative.   If one likes, the operator $\Def^\ast \Def$ can be replaced by the Hodge Laplacian, $\dd\dd^\ast+\dd^\ast \dd$, or the Bochner Laplacian $\nabla^\ast \nabla$.  They are related by the following formula
\be\label{def}
2\Def^\ast \Def =\nabla^\ast \nabla+\dd\dd^\ast-\Ric= \dd\dd^\ast+\dd^\ast \dd+\dd\dd^\ast- 2 \Ric.
\ee
It will be clear from the proof of Theorem \ref{MAINTHM} that the statement of the result holds for these operators as well.  Hence we have
\begin{cor}\label{cor1}
Theorem \ref{MAINTHM} remains valid if the operator $2 \Def^* \Def u$ in \eqref{NSequation} is replaced by $L=\nabla^\ast\nabla$ or by $L=\dd\dd^\ast+\dd^\ast \dd$.
\end{cor}
The statement of Theorem \ref{MAINTHM} and Corollary \ref{cor1} can be extended to more general manifolds.  We do this in section \ref{generalM}.  The next section, Section 2, gathers all the necessary tools, and in Section 3, we bring everything together to establish Theorem \ref{MAINTHM} and Corollary \ref{cor1}.

\section{Preliminaries}
We use the following spaces.  Let $M$ be a complete Riemannian manifold of dimension $N$.  Then
\begin{itemize}
\item $\Lambda^k (M)$ denotes the space of smooth $k$-forms on $M$;
\item $\Lambda^k_c(M)$ denotes the space of smooth $k$-forms with compact support on $M$;
\item $\Lambda^1_{c,\sigma}(M)$ is the space of all smooth, $\dd^*$-closed (co-closed), compactly supported 1-forms on $M$;
\end{itemize}
\subsection{Key Lemmas}
We begin with the following simple observation.

\begin{lemma}[From $\dot H^1$ to $L^2$]\label{VeryGoodLemma}
Let $N \geq 2$, and consider a smooth $1$-form $u \in \Lambda^1(\mathbb{H}^N(-a^2))$ which satisfies  \begin{equation}\label{Louville1}
\int_{\mathbb{H}^N(-a^2)} \big | \nabla u \big |^2 \Vol_{\mathbb{H}^N(-a^2)} < \infty .
\end{equation}
Then $u$ also satisfies
\begin{itemize}
\item $\dd u \in L^2(\mathbb{H}^N(-a^2))$,
\item $u \in L^2(\mathbb{H}^N(-a^2))$,
\item $\dd^* u \in L^2(\mathbb{H}^N(-a^2))$.
\end{itemize}
Moreover, we have the following a priori estimate
\begin{equation}\label{aprioriestimate}
\int_{\mathbb{H}^N(-a^2)} \big | u \big |^2 \Vol_{\mathbb{H}^N(-a^2)} \leq \frac{2}{(N-1)a^2} \int_{\mathbb{H}^N(-a^2)} \big | \nabla u \big |^2 \Vol_{\mathbb{H}^N(-a^2)} .
\end{equation}
\end{lemma}
\begin{proof}
Consider a smooth $1$-form $u \in \Lambda^1(\mathbb{H}^N(-a^2))$ which satisfies \eqref{Louville1}.  Recall the following expressions of $\dd u$ and $\dd^* u$ in terms of $\nabla u$
\begin{equation}\label{goodexpressions}
\begin{split}
\dd u & = \eta^\alpha \wedge \nabla_{e_{\alpha}} u , \\
\dd^* u & = \iota_{e_{\alpha}} \nabla_{e_{\alpha}} u ,
\end{split}
\end{equation}
where $\{e_{\alpha} : 1 \leq \alpha \leq N \}$ is a local orthonormal frame of $T\mathbb{H}^N(-a^2)$, and $\{\eta^{\alpha} : 1 \leq \alpha \leq N \}$ is the associated dual local frame of $T^* \mathbb{H}^N(-a^2)$.   Equivalently, we can write in coordinates
\begin{equation}\label{goodexpressions2}
\begin{split}
\dd u & = \frac 12(\partial_i u_j -\partial_j u_i)\dd x^i \wedge \dd x^j=  \frac 12(\nabla_i u_j -\nabla_j u_i)\dd x^i \wedge \dd x^j\quad\mbox{since} \ \Gamma^l_{ij}=\Gamma^l_{ji},\\
\dd^* u & =-\nabla^j u_j,
\end{split}
\end{equation}
where we sum over repeated indices.
Then \eqref{goodexpressions}  or  \eqref{goodexpressions2} immediately leads to the following two estimates
\begin{equation}
\begin{split}
\big \| \dd u \big \|_{L^2(\mathbb{H}^N(-a^2))} & \leq 2^{\frac{1}{2}} \big \| \nabla u  \big \|_{L^2(\mathbb{H}^N(-a^2))} ,\\
\big \| \dd^* u \big \|_{L^2(\mathbb{H}^N(-a^2))} & \leq N \big \| \nabla u  \big \|_{L^2(\mathbb{H}^N(-a^2))} .
\end{split}
\end{equation}
Next, we proceed to prove that $u \in L^2(\mathbb{H}^N(-a^2))$. Since $\Ric w=-(N-1)a^2 w$ on $\mathbb{H}^N(-a^2)$, by the Weitzenb\"ock formula
\begin{equation}\label{W}
\nabla^* \nabla u = \dd \dd^* u + \dd^* \dd u - \Ric u ,
\end{equation}
we have
\begin{equation}\label{KEY}
\nabla^* \nabla u = \dd \dd^* u + \dd^* \dd u  + (N-1)a^2 u .
\end{equation}
Next, take a preferred point of reference $O \in \mathbb{H}^N(-a^2)$. For any $R >1$, consider a bump function $\phi_R \in C^{\infty}_c (\mathbb{H}^N(-a^2))$ which satisfies
\begin{equation}\label{Bumpfunction}
\begin{split}
& \chi_{B_O(R)} \leq \phi_R \leq \chi_{B_O(2R)} , \\
& \big | \nabla \phi_R \big | = \big | \dd \phi_R  \big | \leq \frac{2}{R} .
\end{split}
\end{equation}
Now integrate \eqref{KEY} against $\phi_R^2u$   to obtain (we drop the notation for $ \Vol_{\mathbb{H}^N(-a^2)} $)
\begin{equation}\label{ONE}
\begin{split}
 \int_{\mathbb{H}^N(-a^2)} g (\nabla^* \nabla u , \phi_R^2 u  ) &=  \int_{\mathbb{H}^N(-a^2)} g (\dd \dd^* u , \phi_R^2 u )  +
\int_{\mathbb{H}^N(-a^2)} g (\dd^* \dd u , \phi_R^2 u ) \\
 &\qquad+ (N-1) a^2 \int_{\mathbb{H}^N(-a^2)} \phi_R^2 |u|^2 .
\end{split}
\end{equation}
Since $\phi_R^2 u$ has compact support in $\mathbb{H}^N(-a^2)$, we can integrate the left hand side of \eqref{ONE} by parts.  Hence
\begin{equation}\label{TWO}
\begin{split}
 \int_{\mathbb{H}^N(-a^2)} g (\nabla^* \nabla u , \phi_R^2 u )
=  \int_{\mathbb{H}^N(-a^2)} g (\nabla u , 2 \phi_R \dd \phi_R \otimes u)
+ \int_{\mathbb{H}^N(-a^2)} \phi_R^2 \big | \nabla u \big |^2   .
\end{split}
\end{equation}
Similarly, from the right hand side of \eqref{ONE} we get (using $\dd^\ast(f u)=-g(\dd f, u)+f\dd^\ast u $, for a function $f$)
\begin{equation}\label{THREE}
\begin{split}
 \int_{\mathbb{H}^N(-a^2)} g (\dd \dd^* u , \phi_R^2 u )
=  \int_{\mathbb{H}^N(-a^2)} -2 \phi_R \cdot \dd^* u \cdot g(\dd \phi_R , u )
+ \int_{\mathbb{H}^N(-a^2)} \phi_R^2  |\dd^* u |^2    ,
\end{split}
\end{equation}
and also that
\begin{equation}\label{FOUR}
\begin{split}
 \int_{\mathbb{H}^N(-a^2)} g(\dd^* \dd u , \phi_R^2 u)
=   \int_{\mathbb{H}^N(-a^2)} g (\dd u , 2 \phi_R \dd \phi_R \wedge u )
+ \int_{\mathbb{H}^N(-a^2)} \phi_R^2 \big | \dd u\big |^2   .
\end{split}
\end{equation}
By combining \eqref{ONE}, \eqref{TWO}, \eqref{THREE}, and \eqref{FOUR}, we yield
\begin{equation}\label{FIVE}
\begin{split}
&\int_{\mathbb{H}^N(-a^2)} g ( \nabla u , 2\phi_R\dd\phi_R\otimes u )  +\int_{\mathbb{H}^N(-a^2)}  \phi^2_R \abs{\nabla u}^2    \\
&=\int_{\mathbb{H}^N(-a^2)} \phi^2_R( \abs{ \dd^* u }^2 + \big | \dd u \big |^2 )
-\int_{\mathbb{H}^N(-a^2)} 2 \phi_R \cdot \dd^* u \cdot g ( \dd \phi_R , u )     \\
&\quad +   \int_{\mathbb{H}^N(-a^2)} 2 \phi_R \cdot g ( \dd u , \dd \phi_R \wedge u    )
+(N-1)a^2 \int_{\mathbb{H}^N(-a^2)} \phi_R^2 |u|^2.  \end{split}
\end{equation}
Rearranging and using $\int_{\mathbb{H}^N(-a^2)} \phi^2_R( \abs{ \dd^* u }^2 + \big | \dd u \big |^2 ) \geq 0$, we have
\begin{equation}\label{FIVEa}
\begin{split}
&(N-1)a^2 \int_{\mathbb{H}^N(-a^2)} \phi_R^2 |u|^2\\
&\leq \int_{\mathbb{H}^N(-a^2)}  \phi^2_R \abs{\nabla u}^2  +\int_{\mathbb{H}^N(-a^2)} g ( \nabla u , 2\phi_R\dd\phi_R\otimes u )   \\
&+\int_{\mathbb{H}^N(-a^2)} 2 \phi_R \cdot \dd^* u \cdot g ( \dd \phi_R , u )
-   \int_{\mathbb{H}^N(-a^2)} 2 \phi_R \cdot g ( \dd u , \dd \phi_R \wedge u    ) .
  \end{split}
\end{equation}
We now estimate the last three terms on the right hand side of \eqref{FIVEa} by the means of the  Cauchy's inequality with $\epsilon$.  We get
\begin{equation}\label{SIX}
\begin{split}
 \bigg |  \int_{\mathbb{H}^N(-a^2)} 2\phi_R \cdot g (\nabla u , \dd \phi_R \otimes u )   \bigg |
 \leq \epsilon \int_{\mathbb{H}^N(-a^2)} \phi_R^2 \big | u \big |^2  +
\frac{1}{\epsilon}\int_{\mathbb{H}^N(-a^2)}  \frac{4}{R^2} \big | \nabla u\big |^2,
\end{split}
\end{equation}
 \begin{equation}\label{SEVEN}
\begin{split}
\bigg |\int_{\mathbb{H}^N(-a^2)} 2 \phi_R \cdot \dd^* u \cdot g ( \dd \phi_R , u )     \bigg |
 \leq \epsilon   \int_{\mathbb{H}^N(-a^2)}  \phi_R^2 \big | u \big |^2 +
\frac{1}{\epsilon} \int_{\mathbb{H}^N(-a^2)}  \frac{4}{R^2} \big | \dd^* u \big |^2  ,
\end{split}
\end{equation}
 \begin{equation}\label{EIGHT}
\begin{split}
\bigg | \int_{\mathbb{H}^N(-a^2)} 2 \phi_R \cdot g ( \dd u , \dd \phi_R \wedge u    )   \bigg |
  \leq \epsilon   \int_{\mathbb{H}^N(-a^2)}  \phi_R^2 \big | u \big |^2  +
\frac{1}{\epsilon} \int_{\mathbb{H}^N(-a^2)}  \frac{4}{R^2} \big | \dd u \big |^2  .
\end{split}
\end{equation}
It follows
\begin{equation}\label{NINE}
\begin{split}
& (N-1) a^2 \int_{\mathbb{H}^N(-a^2)} \phi_R^2 \big | u \big |^2 \\
&\leq  \int_{\mathbb{H}^N(-a^2)} \phi_R^2 \big | \nabla u \big |^2 +3 \epsilon \int_{\mathbb{H}^N(-a^2)} \phi_R^2 \big | u \big |^2
+ \frac{4}{\epsilon R^2} \int_{\mathbb{H}^N(-a^2)}  \bigg ( \big | \nabla u \big |^2 + \big | \dd u \big |^2 + \big | \dd^* u \big |^2    \bigg )  \\
& \leq  \int_{\mathbb{H}^N(-a^2)} \phi_R^2 \big | \nabla u \big |^2 +3 \epsilon \int_{\mathbb{H}^N(-a^2)} \phi_R^2 \big | u \big |^2
+ \frac{4(N^2+3)}{\epsilon R^2} \int_{\mathbb{H}^N(-a^2)}   \big | \nabla u \big |^2    .
\end{split}
\end{equation}
Now, if $\epsilon = \frac{(N-1)a^2}{6}$, we obtain for all $R > 0 $
\begin{equation}
\frac{(N-1)a^2}{2} \int_{\mathbb{H}^N(-a^2)} \phi_R^2 \big | u \big |^2   \leq \bigg ( 1 + \frac{4(N^2+3)}{R^2} \frac{6}{(N-1)a^2} \bigg ) \int_{\mathbb{H}^N(-a^2)} \big | \nabla u \big |^2   .
\end{equation}
So, by taking the limit on both sides of the above estimate, it follows from the dominated convergence theorem that the following estimate holds
\begin{equation}
\frac{(N-1)a^2}{2} \int_{\mathbb{H}^N(-a^2)} \big | u \big |^2
\leq \int_{\mathbb{H}^N(-a^2)} \big | \nabla u \big |^2    .
\end{equation}
\end{proof}

% %%%%%%%%%%%%%%%%%%%%%%
%%%%%                         END OF FIRST LEMMA
%%%%%%%%%%%%%%%%%%%%%%%%%%%%%

%The estimate \eqref{aprioriestimate} in Lemma \ref{VeryGoodLemma} gives an estimate of the $L^2$-norm of a smooth vector field $u$ in terms of $\| \nabla u \|_{L^2(\mathbb{H}^N(-a^2))}$, provided the latter is finite.  From there we immediately get the exact relationship between $\| \Def u \|_{L^2(\mathbb{H}^N(-a^2))}$ and $\| u \|_{L^2(\mathbb{H}^N(-a^2))}$.
%
%\end{proof}

\begin{lemma}\label{GoodLemmaTWO}
Consider a smooth $1$-form $u$ on $\mathbb{H}^N (-a^2)$, which satisfies the following finite Dirichlet norm property
\begin{equation}\label{Condition}
\int_{\mathbb{H}^N(-a^2)} \big | \nabla u \big |^2 \Vol_{\mathbb{H}^N(-a^2)} < \infty .
\end{equation}
Then it follows that the following identity holds
\begin{equation}\label{EASYEASYEASY}
\begin{split}
\big \| \Def u \big \|_{L^2( \mathbb{H}^N(-a^2))}^2 &=  \big \| \dd^* u \big \|_{L^2( \mathbb{H}^N(-a^2))}^2 +
\frac 12\norm {\dd u }_{L^2( \mathbb{H}^N(-a^2))}^2 \\
&\qquad+ (N-1)a^2 \big \| u \big \|_{L^2( \mathbb{H}^N(-a^2))}^2 .
\end{split}
\end{equation}
\end{lemma}
\begin{proof}
Let $u$ be a smooth vector field on $\mathbb{H}^N(-a^2)$ which satisfies condition \eqref{Condition}. By Lemma \ref{VeryGoodLemma}, we have $u \in L^2(\mathbb{H}^N(-a^2))$.  Next, as in the proof of Lemma \ref{VeryGoodLemma}, we consider for each $R > 0$ a bump function $\phi_R \in C^{\infty}(\mathbb{H}^N (-a^2))$, which satisfies the conditions in \eqref{Bumpfunction}. Then, it is easy to verify that we have the following property
\begin{equation}\label{H10}
\lim_{R \rightarrow + \infty} \big \| \phi_R u - u  \big \|_{H^1(\mathbb{H}^N(-a^2))} = 0 ,
\end{equation}
which simply tells us that $u \in H^1_0(\mathbb{H}^N(-a^2))$. As a result, we can now find a sequence $\{w_k\}_{k=1}^{\infty}$ in $\Lambda_c^1(\mathbb{H}^N(-a^2))$ such that
\begin{equation}\label{limiting}
\lim_{k\rightarrow + \infty } \bigg ( \big \| w_k - u \big \|_{L^2(\mathbb{H}^N(-a^2))}^2 + \big \| \nabla (w_k - u ) \big \|_{L^2(\mathbb{H}^N(-a^2))}^2 \bigg ) = 0 .
\end{equation}
From \eqref{def}, which holds for any smooth $1$-form $w$,
%and $\Ric w=-(N-1)a^2 w$ on $\mathbb{H}^N(-a^2)$,
we have
\begin{equation}
2 \Def^* \Def w = 2 \dd \dd^* w + \dd^* \dd w + 2(N-1)a^2 w .
\end{equation}
By applying the above identity to each $w_k$ and performing integration by parts, we deduce
\begin{equation}\label{EASYIdentity}
\begin{split}
2 \big \| \Def w_k \big \|_{L^2(\mathbb{H}^N(-a^2))}^2 &= 2 \big \| \dd^* w_k \big \|_{L^2(\mathbb{H}^N(-a^2))}^2 + \big \| \dd w_k \big \|_{L^2(\mathbb{H}^N(-a^2))}^2\\
&\quad+ 2(N-1)a^2 \big \| w_k \big \|_{L^2(\mathbb{H}^N(-a^2))}^2 .
\end{split}
\end{equation}
Due to the pointwise estimate $\big | \Def w \big | \leq \big | \nabla w \big | $, which holds on $\mathbb{H}^N(-a^2)$ for any smooth $1$-form $w$, it follows that we have
\begin{equation}
\begin{split}
\abs{\big \| \Def w_k \big \|_{L^2(\mathbb{H}^N(-a^2))} - \big \| \Def u \big \|_{L^2(\mathbb{H}^N(-a^2))}   } & \leq
\big \| \Def (w_k - u) \big \|_{L^2(\mathbb{H}^N(-a^2))} \\
& \leq \big \| \nabla (w_k - u ) \big \|_{L^2(\mathbb{H}^N(-a^2))} .
\end{split}
\end{equation}
In light of property \eqref{limiting}, we can pass to the limit in the above inequality and deduce
\begin{equation}\label{limitingTWO}
\lim_{k\rightarrow + \infty } \big \| \Def w_k \big \|_{L^2(\mathbb{H}^N(-a^2))} = \big \| \Def u \big \|_{L^2(\mathbb{H}^N(-a^2))} .
\end{equation}
By essentially the same kind of reasoning, we also get the following limiting properties
\begin{equation}\label{limitingTHREE}
\begin{split}
\lim_{k\rightarrow + \infty } \big \| \dd^* w_k \big \|_{L^2(\mathbb{H}^N(-a^2))} & = \big \| \dd^* u \big \|_{L^2(\mathbb{H}^N(-a^2))}, \\
\lim_{k\rightarrow + \infty } \big \| \dd w_k \big \|_{L^2(\mathbb{H}^N(-a^2))} & = \big \| \dd u \big \|_{L^2(\mathbb{H}^N(-a^2))} .
\end{split}
\end{equation}
Then by  \eqref{limitingTWO} and \eqref{limitingTHREE}, we can now take $k \rightarrow \infty$ in  \eqref{EASYIdentity}  and deduce that identity \eqref{EASYEASYEASY} must hold for $u$. This completes the proof of Lemma \ref{GoodLemmaTWO}.
\end{proof}

\begin{lemma}[Interpolation Estimates]

Let $u \in \Lambda^1(\mathbb{H}^N(-a^2))$ be smooth and satisfy
\begin{equation}\label{h1dot}
\int_{\mathbb{H}^N(-a^2)} \big | \nabla u \big |^2 \Vol_{\mathbb{H}^N(-a^2)} < \infty
\end{equation}
for $N\geq 2$.  If $N\geq 5$, in addition assume $u \in  L^\infty$.  Then $u \in L^p(\mathbb{H}^N(-a^2))$ for $p=3, 4$.  More precisely, we have
\be\label{L3}
\begin{split}
\norm{u}_{L^3(\mathbb{H}^N(-a^2))}&\lesssim \norm{u}_{H^1(\mathbb{H}^N(-a^2))}, \qquad 2\leq N\leq 6,\\
\norm{u}_{L^3(\mathbb{H}^N(-a^2))}&\lesssim \norm{u}^\frac 23_{H^1(\mathbb{H}^N(-a^2))}\norm{u}^\frac 13_{L^\infty(\mathbb{H}^N(-a^2))}, \quad N\geq 7.\\
\end{split}
\ee
And
\be\label{L4}
\begin{split}
\norm{u}_{L^4(\mathbb{H}^N(-a^2))}&\lesssim \norm{u}_{H^1(\mathbb{H}^N(-a^2))}, \quad N=2, 3, 4,\\
\norm{u}_{L^4(\mathbb{H}^N(-a^2))}&\lesssim \norm{u}^\frac 12_{H^1(\mathbb{H}^N(-a^2))}\norm{u}^\frac 12_{L^\infty(\mathbb{H}^N(-a^2))}, \quad N\geq 5.
\end{split}
\ee
\end{lemma}
\begin{proof}  The $L^4$ estimate for $N=2$ follows from the Ladyzhenskaya inequality on $\mathbb{H}^N(-a^2)$ (see \cite[Lemma 2.11]{CC13}).  For $N=3$, interpolate between $L^2$ and $L^6$ and use the Sobolev embedding $H^1 \hookrightarrow L^6$.  For $N=4$ this is just the Sobolev embedding $H^1 \hookrightarrow L^4$ and for $N\geq 5$ it follows by interpolation between $L^2$ and $L^\infty$.  Similarly,  $L^3$ estimates follow by interpolation between $L^2$ and $L^{\frac{2N}{N-2}}$ for $2\leq N\leq 5$ and Sobolev embedding $H^1\hookrightarrow L^{\frac{2N}{N-2}}$.  For $N=6$ this is just the Sobolev embedding  $H^1\hookrightarrow L^3$. For $N\geq 7$, the $L^3$ estimates follow by interpolation between $L^2$ and $L^\infty$.
 \end{proof}
\subsection{Currents}\label{currents}
We briefly recall some basic language about currents from \cite{DeRhamEng}. Let $M$ be an $N-$dimensional manifold, and $\Lambda^k_c(M)$ denote the space of smooth $k$-forms that are compactly supported in $M$.  Then a current $T$ is a linear functional on $\Lambda^k_c(M)$, with the action denoted by $T[\phi]$ for $\phi\in \Lambda^k_c(M)$ \cite[p.34]{DeRhamEng}.
%
%\end{defn}
%A relevant example is an analog of $f \in L^1_{loc}$ giving a rise to a distribution:
For example, if $\alpha$ is a locally integrable $(N-k)$-form, then
\be\label{Talpha}
T_\alpha[\phi]=\int_M \alpha \wedge \phi.
\ee
The scalar product on forms is defined by
\be\label{p2b1}
(w,v)=\int_M g(w,v) \Vol_M=\int_M w \wedge \ast v \Vol_M.
\ee
Note that
\[
(w,v)=T_w[\ast v].
\]
Then a scalar product of a current $T$ with a form $v$ \cite[p.102]{DeRhamEng} is
\be\label{p2b}
(T,v)=T[\ast v].
\ee
If $v$ is compactly supported, then we have \cite[p.105]{DeRhamEng}
\be\label{p2b2}
(\dd T,v)=(T,\dd^\ast v), \quad (\dd^\ast T,v)=(T,\dd v).
\ee
Next we recall a lemma from \cite{CC13}, which rephrases \cite[Thm 17']{DeRhamEng}.
\begin{lemma}\label{p2}
 Let $T$ be a current of degree $1$.
Then $(T,v)=0$ for all
$v \in \Lambda^1_{c,\sigma}(M)$ if and only if $T=dP$ for some $0$ degree current $P$.
\end{lemma}
\subsection{Functional analysis framework.}\label{FunctAnalysis}

Now consider the linear space $\textbf{V}$ defined by
\begin{equation}
\textbf{V} = \overline{\Lambda_{c, \sigma}^1 (\mathbb{H}^N(-a^2))}^{H^1} ,
\end{equation}
where $\Lambda_{c, \sigma}^1 (\mathbb{H}^N(-a^2))$ is the space of all smooth, $\dd^*$-closed, compactly supported $1$-forms on $\mathbb{H}^N(-a^2)$.
So $\textbf{V}$ is the completion of $\Lambda_{c, \sigma}^1 (\mathbb{H}^N(-a^2))$ in the $H^1$ norm. We now state and prove the following result, which was obtained in \cite{CC13} in 2D.
\begin{thm}\label{decomp}
Let $N = 2$. Then
\begin{equation}\label{decomp1}
\big \{ v \in H^1_0(\H) : \dd^* v = 0 \big \} = \textbf{V} \oplus \mathbb{F} ,
\end{equation}
where
\[
\mathbb{F} = \{ \alpha \in L^2(\H) : \alpha=\dd F,\  F \ \mbox{is a harmonic function on} \  \H \} .
\]
If $N \geq 3$, then we have
\begin{equation}\label{decomp2}
\big \{ v \in H^1_0(\mathbb{H}^N(-a^2)) : \dd^* v = 0 \big \} = \textbf{V}.
\end{equation}
\end{thm}
\begin{remark} The main idea of the proof is the same as in 2D.  For completeness, we reproduce the main details.  The first difference in the proof comes towards the middle, where we cannot use the Hodge $\ast$ operator to identify a $2-$form with a function.  Finally, the main difference for $N\geq 3$ is in the end, where by \cite{Dodziuk} we know that $\F\equiv \{0\}$, and hence we obtain  \eqref{decomp2} instead of \eqref{decomp1}.
%However, at this point, we are missing a crucial ingredient.  Namely the fact that $\F \subset H^1$.  In 2D this was showed also in \cite{CC13} using a simple argument from complex analysis.  For now, we assume this and leave the proof for the end of this section.
\end{remark}

\begin{proof}
$\big \{ v \in H^1_0(\mathbb{H}^N(-a^2)) : \dd^* v = 0 \big \}$ is a Hilbert space, when equipped with the following standard inner product
\begin{equation}
[u, v] = \int_{\mathbb{H}^N(-a^2)} g(u, v) \Vol_{\mathbb{H}^N(-a^2)} + \int_{\mathbb{H}^N(-a^2)} g(\nabla u, \nabla v) \Vol_{\mathbb{H}^N(-a^2)} .
\end{equation}
Clearly $\textbf{V}$ is a closed subspace of $\big \{ v \in H^1(\mathbb{H}^N(-a^2)) : \dd^* v = 0 \big \}$.
We have the orthogonal decomposition
\begin{equation}
\big \{ v \in H^1_0(\mathbb{H}^N(-a^2)) : \dd^* v = 0 \big \} = \textbf{V} \oplus \textbf{V}^{\perp} ,
\end{equation}
where $\textbf{V}^{\perp}$ is the orthogonal complement of $\textbf{V}$ in $\big \{ v \in H^1_0(\mathbb{H}^N(-a^2)) : \dd^* v = 0 \big \}$.  We show $\textbf{V}^{\perp} \subset \mathbb{F}$, and because by \cite{Dodziuk}, $\F\equiv\set{0}$ for $N\geq 3$, we get \eqref{decomp2}.  For $N=2$, we get $\textbf{V}^\perp=\F$.

Let $v \in \textbf{V}^{\perp}$.  By definition, $\dd^* v = 0$, and
\begin{equation}\label{ELEVEN}
[v,\theta] = 0,
\end{equation}
for $\theta \in \textbf V$ and in particular, for any test $1$-form $\theta \in \Lambda_{c, \sigma}^1(\mathbb{H}^N(-a^2))$.  Again, by \eqref{KEY} we have
\begin{equation}
\nabla^*\nabla \theta = \dd^* \dd \theta + (N-1)a^2 \theta.
\end{equation}
So by integration by parts
\begin{equation*}
\begin{split}
\int_{\mathbb{H}^N(-a^2)} g (\nabla v, \nabla \theta ) \Vol_{\mathbb{H}^N(-a^2)} & = \int_{\mathbb{H}^N(-a^2)} g (v, \nabla^* \nabla \theta ) \Vol_{\mathbb{H}^N(-a^2)} \\
& = \int_{\mathbb{H}^N(-a^2)} g (v , \dd^* \dd \theta ) + (N-1)a^2 g (v, \theta ) \Vol_{\mathbb{H}^N(-a^2)} .
%& = \big <  \dd^* \dd v , \theta   \big > + (N-1)a^2 g (v, \theta ) \Vol_{\mathbb{H}^N(-a^2)} ,
\end{split}
\end{equation*}
Now, using the language of currents, we can write the last line as $(\dd^\ast \dd v, \theta)+ (N-1)a^2(v,\theta)$, where $(\cdot, \cdot)$ is the scalar product of a current with a form.  So \eqref{ELEVEN} is equivalent to
\[
(\dd^\ast \dd v, \theta)+ ((N-1)a^2+1)(v,\theta)=0,
\]
or
\[
\left(\dd^\ast \dd v +((N-1)a^2+1)v,\theta\right)=0
\]
for all $\theta \in \Lambda_{c, \sigma}^1(\mathbb{H}^N(-a^2))$.
Thus by Lemma \ref{p2}
\begin{equation}\label{Stokesequation}
\dd^* \dd v + \big ( (N-1)a^2 + 1 \big ) v + \dd P = 0,
\end{equation}
for some $0$ degree current $P$.

Now, by applying $\dd$ on both sides of \eqref{Stokesequation}, we obtain the following equation is satisfied by the $2$-form $\omega = \dd v$
\begin{equation}\label{Vorticityequation}
\dd \dd^* \omega + \big ( (N-1)a^2 + 1 \big ) \omega = 0 .
\end{equation}
We note that so far \eqref{Vorticityequation} holds on $\mathbb{H}^N(-a^2)$ in the sense of currents.

Next, observe that the property $\omega \in L^2(\mathbb{H}^N(-a^2))$ follows directly from the fact that $v \in H^1(\mathbb{H}^N(-a^2))$. Observe that \eqref{Vorticityequation} is an elliptic system. So the standard elliptic theory tells us that the 2-form $\omega$, as a solution to \eqref{Vorticityequation}, must be smooth on $\mathbb{H}^N(-a^2)$.  Thus \eqref{Vorticityequation} actually holds in the classical sense.

We now proceed to prove that the $2$-form $\omega$ is identically zero on $\mathbb{H}^N(-a^2)$. To achieve this, we use the cut-off function $\phi_R$ with properties \eqref{Bumpfunction}.  Integrating \eqref{Vorticityequation} against $\omega \phi_R^2$ immediately gives
\begin{equation}\label{CUTOFFINTEGRAL}
\begin{split}
&\int_{\mathbb{H}^N(-a^2)} g (\dd \dd^* \omega , \omega \phi_R^2  ) \Vol_{\mathbb{H}^N(-a^2)}\\
&\qquad + \big ( (N-1)a^2 +1 \big ) \int_{\mathbb{H}^N(-a^2)} g(\omega , \omega) \phi_R^2 \Vol_{\mathbb{H}^N(-a^2)} = 0.
\end{split}
\end{equation}
Integrating by parts in the first term above will produce an expression $$\dd^* \big (\phi_R^2 \omega \big ) = (-1)^{3N+1} * \dd * \big ( \phi_R^2 \omega \big )=(-1)^{3N+1} * \dd  (\phi_R^2 \ast \omega)=(-1)^{3N+1} *( \dd \phi_R^2\wedge \ast \omega)+\phi_R^2\dd^\ast \omega.$$
If we also use that
\begin{equation*}
\int_{\mathbb{H}^N(-a^2)} g (*\alpha , * \beta )  \Vol_{\mathbb{H}^N(-a^2)} = \int_{\mathbb{H}^N(-a^2)} g (\alpha , \beta )  \Vol_{\mathbb{H}^N(-a^2)} ,
\end{equation*}
then the first term in \eqref{CUTOFFINTEGRAL} becomes
\begin{align}
\int_{\mathbb{H}^N(-a^2)} g (\dd \dd^* \omega , \omega \phi_R^2  ) \Vol_{\mathbb{H}^N(-a^2)}&=(-1)^{3N+1}\int_{\mathbb{H}^N(-a^2)} g (\dd^* \omega ,  *( \dd \phi_R^2\wedge \ast \omega)  ) \Vol_{\mathbb{H}^N(-a^2)}\nonumber\\
&\qquad+\int_{\mathbb{H}^N(-a^2)} g (\dd^* \omega , \dd^\ast \omega) \phi_R^2\Vol_{\mathbb{H}^N(-a^2)}\nonumber\\
&=\int_{\mathbb{H}^N(-a^2)} g (\dd\ast \omega ,   \dd \phi_R^2\wedge \ast \omega ) \Vol_{\mathbb{H}^N(-a^2)}\nonumber\\
&\qquad+\int_{\mathbb{H}^N(-a^2)} g (\dd^* \omega , \dd^\ast \omega) \phi_R^2\Vol_{\mathbb{H}^N(-a^2)}.\label{2ndterm}
\end{align}

Everything on the lhs of \eqref{CUTOFFINTEGRAL} has a positive sign except possibly
$\int_{\mathbb{H}^N(-a^2)} g (\dd\ast \omega ,   \dd \phi_R^2\wedge \ast \omega ) \Vol_{\mathbb{H}^N(-a^2)}$, but it can be bounded using Cauchy's inequality with $\eps=\frac 12$ (similarly as in Lemma \ref{VeryGoodLemma})
\begin{align}
\int_{\mathbb{H}^N(-a^2)} g (\dd\ast \omega , &  \dd \phi_R^2\wedge \ast \omega ) \Vol_{\mathbb{H}^N(-a^2)}   \leq  2 \bigg | \int_{\mathbb{H}^N(-a^2)}  g ( \phi_R \cdot \dd * \omega  , \dd \phi_R \wedge * \omega   ) \Vol_{\mathbb{H}^N(-a^2)}     \bigg | \nonumber\\
%& \leq \epsilon \int_{\mathbb{H}^N(-a^2)} \phi_R^2 \cdot g ( \dd * \omega , \dd * \omega  ) \Vol_{\mathbb{H}^N(-a^2)}
%+ \frac{1}{\epsilon} \cdot \frac{4}{R^2}   \int_{\mathbb{H}^N(-a^2)} g (*\omega , * \omega )  \Vol_{\mathbb{H}^N(-a^2)} \\
& \leq \frac 12 \int_{\mathbb{H}^N(-a^2)} \phi_R^2 \abs {\dd^\ast \omega}^2 \Vol_{\mathbb{H}^N(-a^2)}
+ \frac{8}{R^2} \big \| \omega \big \|_{L^2(\mathbb{H}^N(-a^2))}^2.\label{Crucialinequality}
\end{align}
 Using \eqref{Crucialinequality} and \eqref{2ndterm} in \eqref{CUTOFFINTEGRAL}
 gives
\begin{equation}\label{FinalGoodinequality}
\begin{split}
 \big ( (N-1)a^2 + 1 \big )  \int_{\mathbb{H}^N(-a^2)} \abs{\omega}^2 \phi_R^2  \Vol_{\mathbb{H}^N(-a^2)}
+& \frac{1}{2}\int_{\mathbb{H}^N(-a^2)} \phi_R^2 \abs{ \dd^* \omega}^2  \Vol_{\mathbb{H}^N(-a^2)} \\
& \leq \frac{8}{R^2}  \big \| \omega \big \|_{L^2(\mathbb{H}^N(-a^2))}^2 .
\end{split}
\end{equation}
Taking a limit on both sides of \eqref{FinalGoodinequality} as $R\rightarrow \infty$, gives
\begin{equation}
\int_{\mathbb{H}^N(-a^2)} \abs{w}^2\Vol_{\mathbb{H}^N(-a^2)} = 0,
\end{equation}
from which it follows that $\dd v = \omega = 0$ holds on $\mathbb{H}^N(-a^2)$.   We also have $\dd^\ast v=0$ and $v \in L^2$, so $v$ is a harmonic $L^2$ $1$-form.  By \cite{Dodziuk}, if $N\geq 3$, $v=0$.  If $N=2$, $v$ can be nontrivial, and we can show $v\in \F$. It can be also showed $\F\subset \textbf{V}^\perp$ (see \cite[Lemmas 3.2, 3.3, 3.6]{CC13}), so $\F=\textbf{V}^\perp$.
\end{proof}

\section{Proof of Theorem \ref{MAINTHM}}
We are now ready to prove the Liouville theorem in the hyperbolic setting.  To begin, let $N\geq 2$ and consider a smooth divergence free vector field $u$ and smooth function $p$ on $\mathbb{H}^N(-a^2)$, which together satisfy equation \eqref{NSequation} and the finite Dirichlet norm property \eqref{FiniteDirichlet}. Then by Lemma \ref{VeryGoodLemma}, $u \in L^2(\mathbb{H}^N(-a^2))$. This means that $u \in H^1(\mathbb{H}^N(-a^2))$, since now we have
\begin{equation}
\big \|u \big \|_{H^1(\mathbb{H}^N(-a^2))}^2 = \int_{\mathbb{H}^N(-a^2)} \big | u \big |^2 + \big |\nabla u \big |^2  \Vol_{\mathbb{H}^N(-a^2)} < \infty .
\end{equation}
We again consider a bump function $\phi_R \in C_c^{\infty}(\mathbb{H}^N(-a^2))$, which satisfies \eqref{Bumpfunction}. Then  \eqref{H10} holds,
%it is not hard to check that  \begin{equation}\label{densityproperty}
%\lim_{R\rightarrow +\infty} \big \| \phi_R u - u \big \|^2_{H^1(\mathbb{H}^N(-a^2))} = 0,
%\end{equation}
so $u \in H^1_0(\mathbb{H}^N(-a^2))$. And since $\dd^* u = 0$, it follows that $u$ lies in
\begin{equation}
\big \{ v \in H^1_0 (\mathbb{H}^N(-a^2)) : \dd^*v = 0  \big \},
\end{equation}
the function space considered in Section \ref{FunctAnalysis}.
\newline\noindent
\textbf{Case 1: $N\geq 3$.}  By \eqref{decomp2}, there exists a sequence $\{ v_k \}_{k=1}^{\infty}$ in $\Lambda_{c,\sigma}^1 (\mathbb{H}^N(-a^2))$ such that
\begin{equation}\label{lim}
\lim_{k\rightarrow \infty } \big \| v_k - u \big \|_{H^1(\mathbb{H}^N(-a^2))} = 0 .
\end{equation}
Since $\dd^* v_k = 0$, from integration by parts we have
\begin{equation}\label{pressurevanishing}
\int_{\mathbb{H}^3(-a^2)} g (\dd P , v_k ) \Vol_{\mathbb{H}^N(-a^2)} = - \int_{\mathbb{H}^N(-a^2)} \dd^* \big ( P \cdot v_k\big ) \Vol_{\mathbb{H}^N(-a^2)} = 0 .
\end{equation}
By testing equation \eqref{NSequation} against $v_k$, it follows, by taking \eqref{pressurevanishing} into account, that the following relation holds
\begin{equation}\label{plainidentity}
2 \int_{\mathbb{H}^N(-a^2)} g (\Def u , \Def v_k  ) \Vol_{\mathbb{H}^N(-a^2)} + \int_{\mathbb{H}^N(-a^2)} g(\nabla_u u , v_k  ) \Vol_{\mathbb{H}^N(-a^2)} = 0 .
\end{equation}
Next by integration by parts (using $ug(u,v_k-u)\in L^1$), and then \eqref{L4} we have
\begin{equation}\label{nonlinearterm}
\begin{split}
\big|\int_{\mathbb{H}^N(-a^2)} g(\nabla_u u , v_k - u ) \Vol_{\mathbb{H}^N(-a^2)} \big|
&=\big |\int_{\mathbb{H}^N(-a^2)} g( u ,\nabla_u( v_k - u) ) \Vol_{\mathbb{H}^N(-a^2)}\big|
\\
& \leq \big \|  u \big \|^2_{L^4}  \norm{\nabla( v_k - u)}_{L^2} \\
& \lesssim \left\{\begin{array}{ccc}
 &\norm{u}^2_{H^1}\norm{v_k -u}_{H^1} , &N=3, 4,\\
 &  \norm{u}_{H^1} \norm{u}_{L^\infty}  \norm{v_k -u }_{H^1},  &N\geq 5.
\end{array}
\right.
\end{split}
\end{equation}
So by taking the limit in \eqref{nonlinearterm} as $k \rightarrow \infty$, it follows that \begin{equation}\label{finalestimate}
\lim_{k\rightarrow \infty } \bigg | \int_{\mathbb{H}^N(-a^2)} g(\nabla_u u , v_k - u ) \Vol_{\mathbb{H}^N(-a^2)} \bigg | = 0 .
\end{equation}
Now, \eqref{finalestimate} enables us to take the limit on both sides of \eqref{plainidentity} as $k \rightarrow \infty$, and deduce that the following property must hold
\begin{equation}\label{OK}
2 \int_{\mathbb{H}^N(-a^2)} g (\Def u , \Def u ) \Vol_{\mathbb{H}^N(-a^2)} + \int_{\mathbb{H}^N(-a^2)} g(\nabla_u u , u  ) \Vol_{\mathbb{H}^N(-a^2)} = 0 .
\end{equation}
However by \eqref{L3}
\[
\int_{\mathbb{H}^N(-a^2)} \big | u |u|^2 \big | \Vol_{\mathbb{H}^N(-a^2)}  =   \big \| u \big \|_{L^3(\mathbb{H}^N(-a^2))}^3 <\infty.
\]
So because $u|u|^2 \in L^1$, it follows (see for example \cite[(A.27)]{CC13})
\begin{equation}
\int_{\mathbb{H}^N(-a^2)} g(\nabla_u u , u  ) \Vol_{\mathbb{H}^N(-a^2)} = - \frac{1}{2} \int_{\mathbb{H}^N(-a^2)} \dd^* \big ( |u|^2 u \big )  \Vol_{\mathbb{H}^N(-a^2)} = 0 .
\end{equation}
As a result, identity \eqref{OK} now reduces down to the following simple relation
\begin{equation}\label{FinalGlory}
\int_{\mathbb{H}^N(-a^2)} g (\Def u , \Def u ) \Vol_{\mathbb{H}^N(-a^2)} = 0 .
\end{equation}
However, Lemma \ref{GoodLemmaTWO} tells us that $u$ must satisfy the following estimate
\begin{equation}
(N-1)a^2 \big \|u \big \|_{L^2(\mathbb{H}^N(-a^2))}^2 \leq \big \| \Def u \big \|_{L^2(\mathbb{H}^N(-a^2))}^2.
\end{equation}
Hence it follows from \eqref{FinalGlory} that we must have $u =0$ on $\mathbb{H}^N(-a^2)$ for $N\geq 3$ as needed.\\
\textbf{Case 2: $N=2$.}
We employ an entirely different approach here. Take $u \in \Lambda^1(\mathbb{H}^2(-a^2))$ to be a smooth solution to \eqref{NSequation} which satisfies property \eqref{FiniteDirichlet}, and which is $L^{\infty}$-bounded on $\mathbb{H}^2(-a^2)$. The key observation is that by taking the operator $\dd$ on both sides of \eqref{NSequation} (after we use \eqref{def} for $2\Def^\ast \Def$), we can deduce the following vorticity equation
\begin{equation}\label{Vorticityin2D}
(-\triangle ) \omega + 2a^2 \omega + \nabla_u \omega = 0,
\end{equation}
where $\omega \in L^2(\mathbb{H}^2(-a^2))$ is the uniquely determined smooth function for which the relation $\dd u = \omega \Vol_{\mathbb{H}^2(-a^2)}$ holds.
As before, we consider the cut-off function $\phi_R$ with property \eqref{Bumpfunction}. Through testing \eqref{Vorticityin2D} against $\omega \phi_R^2$, we can carry out an integration by parts argument to deduce that
\begin{equation*}
\begin{split}
& \int_{\mathbb{H}^2(-a^2)} \big | \nabla \omega \big |^2 \phi_R^2  + \int_{\mathbb{H}^2(-a^2)} 2a^2 \omega^2 \phi_R^2 \\
\leq & 2 \bigg | \int_{\mathbb{H}^2(-a^2)} \phi_R \omega g(\nabla \omega , \nabla \phi_R )  \bigg |
+ \bigg |  \int_{\mathbb{H}^2(-a^2)} \omega^2 \phi_R g(u , \nabla \phi_R)     \bigg | \\
\leq & \frac{1}{2} \int_{\mathbb{H}^2(-a^2)} \big | \nabla \omega \big |^2 \phi_R^2 + 2 \int_{\mathbb{H}^2(-a^2)} \omega^2 \big | \nabla \phi_R \big |^2 + \frac{2\|u\|_{L^{\infty}}}{R} \int_{\mathbb{H}^2(-a^2)} \omega^2 ,
\end{split}
\end{equation*}
from which we get
\begin{equation}\label{QuiteTrivial}
 \frac{1}{2} \int_{\mathbb{H}^2(-a^2)} \big | \nabla \omega \big |^2 \phi_R^2 + \int_{\mathbb{H}^2(-a^2)} 2a^2 \omega^2 \phi_R^2
\leq  \bigg ( \frac{8}{R^2} +  \frac{2\|u\|_{L^{\infty}}}{R} \bigg )\int_{\mathbb{H}^2(-a^2)} \omega^2.
\end{equation}
By passing to the limit on both sides of \eqref{QuiteTrivial} as $R\rightarrow +\infty$, we deduce that we must have $\|\omega\|_{L^2(-a^2)} = 0$. Hence, it follows that $\dd u = 0$ on $\mathbb{H}^2(-a^2)$. As a result, $u = \dd F$ for some harmonic function $F$ on $\mathbb{H}^2(-a^2)$ as needed.

This concludes the proof of Theorem \ref{MAINTHM}.   We now prove Corollary \ref{cor1}
 \subsection{Proof of Corollary \ref{cor1} }
 The proof is identical except that equation \eqref{NSequation} is replaced either by
 \be\label{e1}
 \nabla^\ast \nabla u +\nabla_uu+\dd p=0
 \ee
 or
 \[
 \dd \dd^\ast u +\dd^\ast \dd u +\nabla_uu+\dd p=0,
 \]
 which simplifies to
  \be\label{e2}
 \dd^\ast \dd u +\nabla_uu+\dd p=0,
 \ee
 since $\dd^\ast u=0$.

 In the case of \eqref{e1}, instead of \eqref{FinalGlory} we obtain $\norm{\nabla u}^2_{L^2(\mathbb H^N(-a^2))}=0$ for $N\geq 3$. If $N=2$, essentially the same argument as in Case 2 of the proof of Theorem \ref{MAINTHM} goes through since $\nabla^\ast\nabla=\dd ^\ast \dd+\dd \dd^\ast -\Ric$.

In the case of \eqref{e2}, we obtain at the end $\norm{\dd u}^2_{L^2(\mathbb H^N(-a^2))}=0$ for $N\geq 3$, which means that $u$ is both closed and co-closed. So $u$ must be a harmonic $L^2$ $1-$form, but since these are trivial on $\mathbb H^N(-a^2)$ for $N\geq 3$, the result follows.  For $N=2$,  the vorticity equation now reads
\begin{equation*}
(-\triangle ) \omega + \nabla_u \omega = 0.
\end{equation*}
This leads to $\|\nabla \omega \|_{L^2(\mathbb{H}^2(-a^2))} = 0$, which tells us that $\omega$ must be a constant on $\mathbb{H}^2(-a^2)$. However, this further implies that $\omega = 0$, since $\omega \in L^2(\mathbb{H}^2(-a^2))$. In this way, we still get $u =\dd F $, with $F$ some harmonic function on $\mathbb{H}^2(-a^2)$.

\section{More general manifolds}\label{generalM}
We proved Theorem \ref{MAINTHM} on $\mathbb H^N(-a^2)$ due to the simplicity of the exposition.   In this section we show how to extend it to more general manifolds.  This can be done if $M$ is a smooth, complete,  $N$-dimensional Riemannian manifold with positive injectivity radius and with
$-(N-1)b^2\leq \Ric \leq -(N-1)a^2<0$.

We first have the following analogs of Lemmas \ref{VeryGoodLemma} and \ref{GoodLemmaTWO}, which we state as corollaries and very briefly sketch out the proofs.

\begin{cor}

Let $N \geq 2$ and let $M$ be a complete, simply connected $N$-dimensional manifold with $\Ric \leq -(N-1)a^2<0$.   Consider a smooth $1$-form $u \in \Lambda^1(M)$ which satisfies  \begin{equation}\label{Louville1M}
\int_{M} \big | \nabla u \big |^2 \Vol_{M} < \infty .
\end{equation}
Then $u$ also satisfies
$\dd u \in L^2(M)$, $u \in L^2(M)$,
$\dd^* u \in L^2(M)$,
and
\begin{equation}\label{aprioriestimate2}
\int_{M} \big | u \big |^2 \Vol_{M} \leq \frac{2}{(N-1)a^2} \int_{M} \big | \nabla u \big |^2 \Vol_{M} .
\end{equation}
\end{cor}

\begin{proof}  Here just like in the proof of Lemma we have \eqref{goodexpressions} and \eqref{goodexpressions2}, so again
\begin{equation}
%\begin{split}
\big \| \dd u \big \|_{L^2(M)}  \leq 2^{\frac{1}{2}} \big \| \nabla u  \big \|_{L^2(M)} ,\quad
\big \| \dd^* u \big \|_{L^2(M)}  \leq  N \big \| \nabla u  \big \|_{L^2(M)} .
%\end{split}
\end{equation}
For $L^2$ bound, we again use the Weitzenb\"ock formula
\begin{equation*}
\nabla^* \nabla u = \dd \dd^* u + \dd^* \dd u - \Ric u,
\end{equation*}
and integrate it against $u$ and a bump function $\phi^2_R \in C^{\infty}_c (M)$ where $\phi_R$  satisfies
\begin{equation}\label{Bumpfunction2}
\begin{split}
 \chi_{B_O(R)} \leq \phi_R \leq \chi_{B_O(2R)} , \quad
 \big | \nabla \phi_R \big | = \big | \dd \phi_R  \big | \leq \frac{2}{R} .
\end{split}
\end{equation}
Using $g (\Ric u , u )=\Ric(u,u)\leq -(N-1)a^2 g(u,u)$, this gives
\begin{equation}\label{ONEa}
\begin{split}
 \int_{M} g (\nabla^* \nabla u , \phi_R^2 u  ) &=  \int_{M} g (\dd \dd^* u , \phi_R^2 u )  +
\int_{M} g (\dd^* \dd u , \phi_R^2 u ) -\int_{M} g(\Ric u, \phi_R^2 u)\\
&\geq  \int_{M} g (\dd \dd^* u , \phi_R^2 u )  +
\int_{M} g (\dd^* \dd u , \phi_R^2 u ) +(N-1) \int_{M} a^2\phi_R^2g (u, u).
\end{split}
\end{equation}
Then the equation \eqref{FIVE} becomes
\begin{equation}\label{FIVEm}
\begin{split}
&\int_{M} g ( \nabla u , 2\phi_R\dd\phi_R\otimes u )  +\int_{M}  \phi^2_R \abs{\nabla u}^2    \\
&\quad\geq\int_{M} \phi^2_R( \abs{ \dd^* u }^2 + \big | \dd u \big |^2 )
-\int_{M} 2 \phi_R \cdot \dd^* u \cdot g ( \dd \phi_R , u )     \\
&\qquad +   \int_{M} 2 \phi_R \cdot g ( \dd u , \dd \phi_R \wedge u    )
+(N-1)a^2 \int_{M} \phi_R^2 |u|^2  \end{split}
\end{equation}
Rearranging and estimating just like in \eqref{FIVEa}-\eqref{EIGHT}, we arrive at \eqref{NINE} stated on $M$ instead of $\mathbb H^N(-a^2)$.  The remainder of the proof is then exactly the same.
  \end{proof}

\begin{cor}\label{GoodLemmaTWOa}
Let $N \geq 2$ and let $M$ be a smooth and complete $N$-dimensional manifold with $\Ric \leq -(N-1)a^2<0$.  Consider a smooth $1$-form on $M$, which satisfies
\begin{equation}\label{ConditionM}
\int_{M} \big | \nabla u \big |^2 \Vol_{M} < \infty .
\end{equation}
Then
\begin{equation}\label{EASYM}
2 \big \| \Def u \big \|_{L^2( M)}^2 \geq 2 \big \| \dd^* u \big \|_{L^2( M)}^2 +
\big \| \dd u \big \|_{L^2( M)}^2 + 2(N-1)a^2 \big \| u \big \|_{L^2( M)}^2 .
\end{equation}
\end{cor}

\begin{proof}
The proof is identical as Lemma \ref{GoodLemmaTWO} except that \eqref{EASYIdentity} becomes
\begin{equation}\label{EASYIdentity2}
2 \big \| \Def w_k \big \|_{L^2(M)}^2 \geq  2 \big \| \dd^* w_k \big \|_{L^2(M)}^2 + \big \| \dd w_k \big \|_{L^2(M)}^2 + 2(N-1)a^2 \big \| w_k \big \|_{L^2(M)}^2 ,
\end{equation}
which leads to \eqref{EASYM}, which in this case is an inequality instead of the identity as in \eqref{EASYEASYEASY}.
 \end{proof}

The interpolation identities are exactly the same, and this is when we need the positive injectivity radius and a lower bound on $\Ric$ (see \cite{Var, Hebey}).  The statements about currents are also the same.  This leaves the statement of Theorem \ref{decomp}.  If we do not include that the manifold is simply connected, then we need to replace the space $\F$ by the space of harmonic $L^2$ 1-forms.  The steps of the proof would be the same, and the conclusion would depend on the fact whether or not the manifold $M$ admits nontrivial harmonic $L^2$ forms.

 \section*{Acknowledgements}
 The first author was partially supported by a grant from the National Science Council of Taiwan (NSC 101-2115-M-009-016-MY2).  The second author was partially supported by a grant from the Simons Foundation \#246255.

\bibliography{ref}

\def\cprime{$'$} \def\cprime{$'$}
\begin{thebibliography}{1}

\bibitem{CC13}
Chi~Hin Chan and Magdalena Czubak.
\newblock {Remarks on the weak formulation of the Navier-Stokes equations on
  the 2D hyperbolic space}.
\newblock {\em arXiv:1309.3496, Submitted}, pages 1--49, September 2013.

\bibitem{DeRhamEng}
Georges de~Rham.
\newblock {\em Differentiable manifolds}, volume 266 of {\em Grundlehren der
  Mathematischen Wissenschaften [Fundamental Principles of Mathematical
  Sciences]}.
\newblock Springer-Verlag, Berlin, 1984.
\newblock Forms, currents, harmonic forms, Translated from the French by F. R.
  Smith, With an introduction by S. S. Chern.

\bibitem{Dodziuk}
Jozef Dodziuk.
\newblock {$L^{2}$} harmonic forms on rotationally symmetric {R}iemannian
  manifolds.
\newblock {\em Proc. Amer. Math. Soc.}, 77(3):395--400, 1979.

\bibitem{Galdi}
Giovanni~P. Galdi.
\newblock {\em An introduction to the mathematical theory of the
  {N}avier-{S}tokes equations. {V}ol. {II}}, volume~39 of {\em Springer Tracts
  in Natural Philosophy}.
\newblock Springer-Verlag, New York, 1994.
\newblock Nonlinear steady problems.

\bibitem{Hebey}
Emmanuel Hebey.
\newblock {\em Nonlinear analysis on manifolds: {S}obolev spaces and
  inequalities}, volume~5 of {\em Courant Lecture Notes in Mathematics}.
\newblock New York University Courant Institute of Mathematical Sciences, New
  York, 1999.

\bibitem{SereginBook}
Gregory Seregin.
\newblock {\em Lecture Notes on Regularity Theory for the {N}avier-{S}tokes
  Equations}.
\newblock World Scientific, 2014.

\bibitem{Var}
N.~Th. Varopoulos.
\newblock Small time {G}aussian estimates of heat diffusion kernels. {I}. {T}he
  semigroup technique.
\newblock {\em Bull. Sci. Math.}, 113(3):253--277, 1989.

\end{thebibliography}
\bibliographystyle{plain}
\end{document}